\documentclass{amsart}[12pt]
\usepackage{amssymb,amsmath}
\usepackage{enumerate}
\usepackage[english]{babel}
\usepackage{color}

\newcommand{\Om} {\Omega}
\newcommand {\ep} {\epsilon}

\newcommand {\gm} {\gamma}
\newcommand {\ii} {\infty}
\newcommand {\dt} {\delta}
\newcommand {\al} {\alpha}
\newcommand {\bt} {\beta}
\newcommand {\lb} {\lambda}
\newcommand {\Lb} {\Lambda}

\newcommand {\sm} {\setminus}
\newcommand {\su} {\subset}
\newcommand {\wt} {\widetilde}
\newcommand {\wh} {\widehat}
\newcommand {\mb} {\mathbf}
\newcommand {\mbb} {\mathbb}

\newcommand {\mc} {\mathcal}

\newtheorem{teo}{Theorem}[section]
\newtheorem{pro}{Proposition}[section]

\newtheorem{lm}{Lemma}[section]

\theoremstyle{definition}
\newtheorem{rem}{Remark}[section]
\newtheorem{df}{Definition}[section]

\title{Noncommutative weighted \\ individual ergodic theorems \\with continuous time}
\begin{document}
\date{September 5, 2018}

\begin{abstract}
We show that ergodic flows in noncommutative fully symmetric spaces (associated with a semifinite von Neumann algebra) generated by continuous semigroups of positive Dunford-Schwartz operators and modulated by bounded Besicovitch almost periodic functions converge almost uniformly. The corresponding local ergodic theorem is also discussed.
\end{abstract}

\author{VLADIMIR CHILIN}
\address{National University of Uzbekistan, Tashkent, Uzbekistan}
\email{vladimirchil@gmail.com; chilin@ucd.uz}

\author{SEMYON LITVINOV}
\address{Pennsylvania State University \\ 76 University Drive \\ Hazleton, PA 18202, USA}
\email{snl2@psu.edu}

\maketitle
\section{Introduction}
In the classical ergodic theory, Besicovitch-weighted individual ergodic theorems have been studied quite extensively
(see, for example, \cite{ry, bo, bl, lot, dj}).

In the noncommutative setting, first individual
ergodic theorem with bounded Besicovitch weights in a von Neumann algebra was obtained in \cite{he}. Later, in \cite{cls},
a similar result concerning the so-called bilaterally almost uniform convergence (in Egorov's sense) was established
in the $L^1$-space associated with a semifinite von Neumann algebra. In \cite{mmt}, utilizing the approach of \cite{jx},
a multi-parameter version of \cite[Theorem 4.6]{cls} was proved for every noncommutative $L^p$-space
with $1< p<\ii$. Recently, almost uniform convergence (in Egorov's sense) of Besicovitch-weighted ergodic averages in fully symmetric spaces of measurable operators was established in \cite[Theorem 4.7 and Sec.\,6]{cl1}.

Note that all of the above were concerned with bounded Besicovitch sequences - generated by the well-studied
Besicovitch almost periodic functions \cite{be} - leaving open the problem what happens in the case of actions of continuous semigroups, when one has to turn to Bisicovitch almost periodic functions. To this end, a noncommutative local ergodic Besicovitch-weighted theorem was first considered in \cite{mk}; see remarks following Theorem \ref{t22}.

Let $\{T_t\}_{t\ge 0}$ be a (continuous) semigroup of positive Dunford-Schwartz operators in a noncommutative
$(L^1+L^\ii)$-space. Our goal is to show that the corresponding ergodic averages modulated by a bounded Besicovitch (zero-Besicovitch) almost periodic function $\bt(t)$, $t\ge 0$, in a noncommutative fully symmetric space converge almost uniformly as $t\to\ii$ (Theorem \ref{t42}) (respectively, as $t\to 0$ (Theorem \ref{t43})). Since the results appear to be new for the commutative setting, a relevant discussion is given in the last section of the article.

\section{Preliminaries}
Let $\mc M$ be a semifinite von Neumann algebra equipped with a faithful normal semifinite trace $\tau$.
Denote by $\mc P(\mc M)$ the complete lattice of projections in $\mc M$.  If $\mathbf 1$ is the identity of $\mc M$ and
$e\in \mc P(\mc M)$, we write $e^{\perp}=\mathbf 1-e$. Also, if $\mc M$ acts in a Hilbert space $\mc H$, then,
given $\{e_\al\}_{\al \in \Lambda}\su \mc P(\mc M)$, denote $\bigwedge\limits_{\al\in \Lambda}e_\al$ the projection
on the subspace $\bigcap\limits_{\al\in \Lambda}e_\al\mc H$. Note that $\bigwedge\limits_{\al\in \Lambda}e_\al
\in \mc P(\mc M)$ and
$$
\tau\bigg(\bigg[\bigwedge_{n=1}^\ii e_n\bigg]^\perp\bigg)\leq \sum_{n=1}^\ii\tau(e_n^\perp)
$$
for any sequence $\{e_n\}\su \mc P(\mc M)$.

Let $L^0=L^0(\mc M,\tau)$ be the $*$-algebra of $\tau$-measurable operators affiliated with $\mc M$, and let $\|\cdot\|_\ii$ be the uniform norm in $\mc M$. Equipped with the {\it measure topology} given by the system
$$
\mc N(\ep,\dt)=\{ x\in L^0: \ \|xe\|_\ii\leq\dt \text{ \ for some \ } e\in\mc P(\mc M) \text{ \ with \ } \tau(e^\perp)\leq\ep\},
$$
$\ep>0$, $\dt>0$, of (closed) neighborhoods of zero, $L^0$ is a complete metrizable topological $*$-algebra \cite{ne}.

Let $L^p=L^p(\mc M,\tau)$, $1\leq p\leq \ii$, ($L^{\ii}=\mc M$)
be the noncommutative $L^p$-space associated with $(\mc M,\tau)$, and let $\| \cdot \|_p$ be the standard norm in
the space $L^p$, $1\leq p< \ii$. For detailed accounts on the noncommutative $L^p$-spaces, $p\in \{0\}\cup [1,\ii)$,
see \cite{se, ne, ye0, px}.

A net $\{ x_\al\}\su L^0$ is said to converge {\it almost uniformly (a.u.)} ({\it bilaterally almost uniformly (b.a.u.)})
to $x\in L^0$  if for any given $\ep>0$ there
is a projection $e\in \mc P(\mc M)$ such that $\tau(e^{\perp})\leq \ep$ and $\| (x-x_\al)e\|_{\ii}\to 0$
(respectively, $\| e(x-x_\al)e\|_{\ii}\to 0$).

It is well-known that if a sequence in $L^0$ converges in measure, then it has a subsequence converging a.u.
Besides, a sequence in $L^0$ converging in $L^p$ for some $1\leq p\leq\ii$ also converges in measure.

A linear operator $T: L^1+\mc M\to  L^1+\mc M$ is called a {\it Dunford-Schwartz operator} (writing $T \in DS$) if
$$
\| T(x)\|_1\leq \| x\|_1 \ \ \forall \ x\in L^1 \text{ \ and \ } \| T(x)\|_{\ii}\leq \| x\|_{\ii} \ \ \forall \ x\in L^{\ii}.
$$
If a Dunford-Schwartz operator $T$ is positive, that is, $T(x)\ge 0$ whenever $x\ge 0$, we will write $T\in DS^+$.
Note that positive absolute contractions in $L^1$, considered in \cite{ye} and then in \cite{cls, mmt, mk},
can be uniquely extended to positive Dunford-Schwartz operators - see \cite{cl}.

Given $x\in L^1+\mc M$ and $T\in DS$, denote
$$
A_n(x)=\frac1n\sum_{k=0}^{n-1}T^k(x).
$$
The following fundamental result is due to Yeadon \cite[Theorem 1]{ye}.

\begin{teo}\label{t11}
Let  $T\in DS^+$. Then for every
$x \in L^1_+$ and $\lb>0$ there exists $e\in \mc P(\mc M)$ such that
$$
\tau(e^\perp)\leq \frac {\| x\|_1} \lb\text{ \ \ and \ \ } \sup_n \| e A_n(x)e \|_\ii \leq \lb.
$$
\end{teo}

Let a semigroup $\{T_t\}_{t\ge 0}\su DS$ be {\it strongly continuous in} $L^1$, that is,
$$
\| T_t(x)-T_{t_0}(x)\|_1\to 0 \text{ \ \ whenever\ \ } t\to t_0 \text{ \ \ for all\ \ } x\in L^1.
$$
Then, given $x\in L^1$ and $t>0$, there exists
$$
A_t(x)=\frac1t\int_0^tT_s(x)ds\in L^1
$$
(see the argument preceding (\ref{e22})).

Here is a continuous extension of Theorem \ref{t11} (cf. \cite[Remark 4.7]{jx}):

\begin{teo}\label{t12}
If $\{T_t\}_{t\ge 0}\su DS^+$ is strongly continuous in $L^1$, then,
given $x\in L^1_+$ and $\lb>0$, there exists $e\in \mc P(\mc M)$ such that
$$
\tau(e^\perp)\leq \frac {2\| x\|_1} \lb
\text{ \ \ and \ \ } \sup_{t> 0} \| eA_t(x)e \|_\ii \leq\lb.
$$
\end{teo}
\begin{proof}
Let $\mbb N$ ($\mbb Q$) be the set of all natural (respectively, rational) numbers, and let
$\frac nm\in \mbb Q$, where $n, m \in \mbb N$. Denote  $y = \int_0^{1}T_{s/m}(x)  ds$.
We have
\begin{equation*}
\begin{split}
0\leq A_{\frac nm}(x)&=\frac mn\int_0^{\frac nm}T_s(x) ds =
\frac1n \int_0^{n}T_{s/m}(x) ds\\
&=\frac 1n\bigg(\int_0^{1}T_{s/m}(x) ds+\dots + \int_{n-1}^{n}T_{s/m}(x) ds\bigg)
=\frac 1n  \sum_{k=0}^{n-1}T_{1/m}^k(y).
\end{split}
\end{equation*}
By Theorem \ref{t11}, given $\lb>0$, there is $f\in \mc P(\mc M)$ such that
$$
\tau(f^\perp)\leq \frac {\| y\|_1}\lb \leq \frac {\| x\|_1}\lb \ \text{ \ \ and \ \ } \sup_n \bigg\| f\,
\frac 1n\sum_{k=0}^{n-1} T_{1/m}^k(y)\, f \bigg\|_\ii\leq \lb,
$$
implying, that
$$
\sup_{0<r\in\mbb Q} \| fA_r(x)f\|_\ii \leq \lb.
$$

If $t>0$ and $0<r_n\to t, \ r_n\in \mbb Q$, then we have $A_{r_n}(x)\to A_t(x)$ in $L^1$, hence in measure. Therefore
 $A_{r_{n_k}}(x)\to A_t(x)$ a.u. for a subsequence $\{r_{n_k}\}\su \{r_n\}$. Thus, it is possible to find
$g\in \mc P(\mc M)$ such that
$$
\tau(g^\perp)\leq \frac {\| x\|_1}\lb \text { \ \ and \ \ } \|gA_{r_{n_k}}(x)g\|_\ii \to \| gA_t(x)g\|_\ii \text{ \ as \ } k\to \ii.
$$
Letting $e=f\land g$, we obtain the desired inequalities.
\end{proof}

\section{Convergence in the space $L^1(\mc M,\mc\tau)$}
Let $\mbb C$ be the field of complex numbers, and let $\mbb C_1= \{z\in \mbb C: |z|=1\}$. A function $p:\mbb R_+\to \mbb C$ is called a {\it trigonomertic polynomial} if $p(t)=\sum\limits_{j=1}^nw_j\lb_j^t$, where $n\in \mbb N$, $\{w_j\}_1^n\su \mbb C$, and $\{\lb_j\}_1^n\su \mbb C_1$.

A Lebesgue measurable function $\bt: \mbb R_+\to \mbb C$ will be called {\it bounded Besicovitch} ({\it zero-Besicovitch}) {\it function} if  $\sup\limits_{t\ge 0}|\bt(t)|<C<\ii$, and for every $\ep>0$ there is a trigonometric polynomial $p_\ep$ such that
\begin{equation}\label{e21}
\limsup_{t\to\ii}\frac 1t\int_0^t|\bt(s)-p_\ep(s)|ds<\ep
\end{equation}
(respectively,
\begin{equation}\label{e21z}
\limsup_{t\to 0}\frac 1t\int_0^t|\bt(s)-p_\ep(s)|ds<\ep\, ).
\end{equation}

Let $\{T_t\}_{t\ge 0}\su DS^+$ be a strongly continuous semigroup in $L^1$, and assume that $\bt:\mbb R_+\to\mbb C$ be a Lebesgue measurable function with $\|\bt\|_\ii<\ii$. Fix $x\in L^1$. Then for any given $y\in\mc M$ the function $\varphi_{x,y}(t)=\tau(T_t(x)y)$ is continuous on $\mbb R_+$. Therefore, if $\mu$ is the Lebesgue measure on $\mbb R_+$, then the map $U_x: \mbb R_+\to  L^1$ defined as $U_x(t) = T_t(x)$ is weakly $\mu$-measurable \cite[Ch.V,\ \S\,4]{yo}. Since, in addition, $U_x(\mbb R)$ is a separable subset in $L^1$, Pettis theorem \cite[Ch.V, \S\,4]{yo} entails that the map $U_x$ is strongly $\mu$-measurable and the real function $\|U_x(t)\|_1=\|T_t(x)\|_1$ is $\mu$-measurable on
$\mbb R_+$. Since $\|T_t(x)\|_1\leq\|x\|_1$, it follows that $\|T_s(x)\|_1$ is an integrable function on $[0, t]$ for any $t>0$. Consequently, $\|\bt(s)T_s(x)\|_1=|\bt(s)|\cdot\|T_s(x)\|_1$ is also integrable on $[0, t]$ for any $t>0$.
By \cite[Ch.V, \S\,5, Theorem 1]{yo}, the function $\bt(s)T_s(x)$ is Bochner $\mu$-integrable on $[0,t]$, $t>0$. Therefore, for   any $x\in L^1$ and $t>0$ there exists
\begin{equation}\label{e22}
B_t(x)=\frac 1t\int_0^t\bt(s)T_s(x)ds \in L^1.
\end{equation}

\begin{lm}\label{l23}
Let $\{T_t\}_{t\ge 0}\su DS^+$ be a strongly continuous semigroup in $L^1$, and let $\bt: \mbb R_+\to \mbb C$ be a Lebesgue measurable function such that $\sup\limits_{t\ge 0}|\bt(t)|\leq C<\ii$. If $x\in L^1$ and $\ep>0$, then there is a projection $e\in \mc P(\mc M)$ satisfying inequalities
$$
\tau(e^\perp)\leq \frac{4\|x\|_1}\ep \text{ \ \ and \ \ } \sup_{t>0}\|eB_t(x)e\|_\ii\leq 48C\ep.
$$
\end{lm}
\begin{proof}
We have $x=(x_1-x_2)+i(x_3-x_4)$, where $x_j\in L^1_+$ and $\|x_j\|_1\leq \|x\|_1$ for each $j=1,2,3,4$.
By Theorem \ref{t12}, given $j$, there exists $e_j\in \mc P(\mc M)$ such that
$$
\tau(e_j^\perp)\leq\frac{\|x_j\|_1}\ep \text{ \ \ and \ \ } \sup_{t>0}\Big\| e_j\, \frac 1t\int_0^tT_s(x_j)ds\, e_j\Big\|_\ii\leq 2\ep.
$$

Next, we have $0\leq \operatorname{Re}\bt(s)+C\leq 2C$ and $0\leq \operatorname{Im}\bt(s)+C\leq 2C, \ s \geq 0$,  implying that
$$
0\leq [\operatorname{Re}\bt(s)+C]T_s(x_j)\leq 2CT_s(x_j) \text{\ \ and\ \ }
0\leq [\operatorname{Im}\bt(s)+C]T_s(x_j)\leq 2CT_s(x_j)
$$
for all $s\ge 0$, hence
$$
0\leq \frac 1t\int_0^t[\operatorname{Re}\bt(s)+C]T_s(x_j)ds\leq 2C\frac 1t\int_0^tT_s(x_j)ds
$$
and
$$
0\leq \frac 1t\int_0^t[\operatorname{Im}\bt(s)+C]T_s(x_j)ds\leq 2C\frac 1t\int_0^tT_s(x_j)ds.
$$
for each $j$. This, together with the decomposition
\begin{equation*}
\begin{split}
B_t(x_j)&=\frac 1t\int_0^t[\operatorname{Re}\bt(s)+C]T_s(x_j)ds+i\frac 1t\int_0^t[\operatorname{Im}\bt(s)+C]T_s(x_j)ds
\\
&-C(1+i)\frac 1t\int_0^tT_s(x_j)ds,
\end{split}
\end{equation*}
implies that
$$
\sup_{t>0}\|e_jB_t(x_j)e_j\|_\ii\leq 12C\ep, \ \ j=1,2,3,4.
$$
Now, letting $e=\bigwedge\limits_{j=1}^4e_j$, we obtain the desired inequalities.
\end{proof}

The following is a noncommutative counterpart of the classical notion of continuity at zero of the maximal operator
(see \cite[Proposition 1.1 and ensuing remarks]{li}).
\begin{df}
Let $(X,\| \cdot\|)$ be a Banach space. A sequence of linear maps $M_n:X\to L^0$ is called {\it bilaterally uniformly
equicontinuous in measure (b.u.e.m.)} at zero if, given $\ep>0$ and $\dt>0$, there exists $\gm>0$ such that
for every $\| x\|<\gm$ there exists a projection $e\in\mc P(\mc M)$ satisfying conditions
$$
\tau(e^\perp)\leq\ep \text{\ \ and\ \ } \sup_n\|eM_n(x)e\|_\ii\leq\dt.
$$
\end{df}

In order to establish a.u. convergence - which is generally stronger than b.a.u. convergence - of the averages $B_t(x)$, we will need the following lemma a proof of which can be found in \cite[Lemma 3.2]{li1}; see also \cite{li2}.

\begin{lm}\label{l24}
Let $M_n: L^1\to L^1$ be sequence of linear maps that is b.u.e.m. at zero. If $\{x_m\}\su L^1$ is such that
$\|x_m\|_1\to 0$, then for every $\ep>0$ and $\dt>0$ there are $e\in \mc P(\mc M)$ and $x_{m_0}\in \{x_m\}$
satisfying conditions
$$
\tau(e^\perp)\leq \ep \text{\ \ and\ \ } \sup_n\|M_n(x_{m_0})e\|_\ii\leq \dt.
$$
\end{lm}

In what follows we shall assume that $\mc M$ has separable predual. Denote by $\nu$ the normalized Lebesgue measure on $\mbb C_1$. Let $\wt{\mc M}$ be the von Neumann algebra of essentially bounded ultraweakly measurable functions $\wt f: (\mbb C_1,\nu)\to \mc M$
equipped with the trace
$$
\wt \tau(\wt f)=\int_{\mbb C_1}\tau(\wt f(z))d\nu(z), \ \ \wt f\ge 0.
$$
Let $\wt{L^1}$ be the Banach space of Bochner $\nu$-integrable functions $\wt f: (\mbb C_1,\nu)\to L^1(\mc M,\tau)$.
As the predual of $\wt{\mc M}$ \cite[Theorem 1.22.13]{sa}, the space $\wt{L^1}$ is isomorphic to
$L^1(\wt{\mc M},\wt\tau)$.

Repeating the argument in \cite[Lemma 2]{da} (see also \cite[Lemma 4.1]{cls}), we obtain the following.
\begin{lm}\label{l21}
If $\wt{L^1}\ni \wt f_t\to \wt f\in \wt{L^1}$ a.u. as $t\to \ii$, then $\wt f_t(z)\to \wt f(z)$ a.u. as $t\to\ii$
for $\nu$-almost all $z\in \mbb C_1$.
\end{lm}

Let $\{T_t\}_{t\ge 0}\su DS^+$ be a strongly continuous semigroup in $L^1$.
Pick $\lb \in \mbb C_1$ and define
$$
T_t^{(\lb)}(\wt f)(z)=T_t(\wt f(\lb^tz)),\ \ \ t\ge 0,\ \ \wt f\in L^1(\wt{\mc M},\wt\tau)+\wt{\mc M}, \ z\in \mbb C_1.
$$
Then it is easily verified that $\{T_t^{(\lb)}\}_{t\ge 0}\su DS^+(\wt{\mc M},\wt\tau)$. Besides, given $t,s\ge 0$,
$\wt f\in L^1(\wt{\mc M},\wt\tau)+\wt{\mc M}$, and $z\in\mbb C_1$, we have
\begin{equation*}
\begin{split}
T_{t+s}^{(\lb)}(\wt f)(z)&=T_{t+s}(\wt f(\lb^{t+s}z))=T_t(T_s(\wt f(\lb^s\lb^tz)))=
T_t(T_s^{(\lb)}(\wt f)(\lb^tz))\\
&=(T_t^{(\lb)}T_s^{(\lb)})(\wt f)(z),
\end{split}
\end{equation*}
that is, $\{T_t^{(\lb)}\}_{t\ge 0}$ is a semigroup.

\begin{pro}\label{p10}
The semigropup  $\{T_t^{(\lb)}\}_{t\ge 0}$ is strongly continuous in $L^1(\wt M,\wt\tau)$.
\end{pro}
\begin{proof}
Let us show that if $0\leq s_n\to 0$ and $\wt f\in L^1(\wt{\mc M},\wt\tau)$, then
$$
\|T_{s_n}^{(\lb)}( \wt f)-\wt f\|_{L^1(\wt{\mc M},\wt\tau)} = \int_{\mbb C_1}\|T_{s_n}(\wt f(\lb^{s_n}z))-\wt f(z)\|_1 \, d \nu(z)\to 0 \ \ \text{as} \ \ n\to\ii.
$$
We have
\begin{equation*}
\begin{split}
\|T_{s_n}(\wt f(\lb^{s_n}z))-\wt f(z)\|_1&\leq \|T_{s_n}(\wt f(\lb^{s_n}z))-T_{s_n}(\wt f(z)) \|_1+
\|T_{s_n}(\wt f(z))-\wt f(z)\|_1\\
&\leq \|\wt f(\lb^{s_n}z)-\wt f(z)\|_1+\|T_{s_n}(\wt f(z))-\wt f(z)\|_1.
\end{split}
\end{equation*}
Since  $\{T_t\}_{t\ge 0}$ is strongly continuous in $L^1$, it follows that $\|T_{s_n}(\wt f(z))-\wt f(z)\|_1\to 0$ for $\nu$-almost all $z \in \mbb C_1$. Besides,
$$
\|T_{s_n}(\wt f(z))-\wt f(z)\|_1 \leq 2 \|\wt f(z)\|_1 \ \ \ \nu-\text{a.e.},
$$
where $\|\wt f(z)\|_1\in L^1(\mbb C_1,\nu)$, which allows us to conclude that
$$
\int_{\mbb C_1}\|T_{s_n}(\wt f(z))-\wt f(z)\|_1 \ d \nu(z) \to 0 \ \ \text{as} \ \  n \to \ii.
$$
Thus, it remains to verify that
$$
\int_{\mbb C_1}\|\wt f(\lb^{s_n}z))-\wt f(z)\|_1\, d \nu(z) \to 0 \ \ \text{as} \ \  n \to \ii.
$$
Let
\begin{equation}\label{e220}
\wt h(z)=\sum\limits_{i=1}^m x_i \chi_{[\lb_i,\lb_{i+1})}(z), \ z \in \mbb C_1, \ x_i \in L^1, \ i=1,\dots,m
\end{equation}
be a simple Bochner measurable function, where $\{\lb_i\}_{i=1}^m$ is a partition of the circle $\mbb C_1$. Then, given
$s\in\mbb R$, we have
$$
\wt h(\lb^sz)=\sum\limits_{i=1}^m x_i \chi_{[\lb^s\lb_i,\lb^s\lb_{i+1})}(z),
$$
where $\{\lb^s\lb_i\}_{i=1}^n$ is another partition of $\mbb C_1$.
Since $s_n \to 0$, it follows that $\lb^{s_n}\lb_i\to \lb_i$ as $n\to\ii$ for all $i=1,\dots,m$, implying that
\begin{equation}\label{e221}
\int_{\mbb C_1}\|\wt h(\lb^{s_n}z)-\wt h(z)\|_1\, d \nu(z) \to 0 \ \ \text{as} \ \  n\to\ii.
\end{equation}

Let $\mc A$ be the subalgebra of the $\sigma$-algebra of Lebesgue measurable sets of $(\mbb C_1,\nu)$ generated by the arcs of the circle $\mbb C_1$. Since any $A\in\mc A$ is  a finite union of pairwise disjoint arcs of $\mbb C_1$, a simple Bochner measurable function
\begin{equation}\label{e222}
\wt h(z)=\sum\limits_{i=1}^m x_i\chi_{A_i}(z), \ A_i\in\mc A,  \ i=1,\dots,m
\end{equation}
has the form (\ref{e220}) for some partition $\{\lb_i\}_{i=1}^l$ of $\mbb C_1$.
Therefore, (\ref{e221}) holds for any simple Bochner measurable function $\wt h$ of the form (\ref{e222}).

Since the $\sigma$-algebra generated by the subalgebra $\mc A$ coincides $\nu$-almost everywhere with the $\sigma$-algebra $\Sigma$ of Lebesgue measurable sets in $(\mbb C_1,\nu)$, given an arbitrary simple Bochner measurable  function
$\wt g(z)=\sum_{i=1}^m x_i\chi_{A_i}(z)$, $A_i\in\Sigma$, $i =1,\dots,m$, there exists a sequence
$\{\wt h_k(z)\}_{k=1}^\ii$ of simple Bochner measurable functions of the form (\ref{e222}) such that
$$
\int_{\mbb C_1}\|\wt g(z)-\wt h_k(z)\|_1\, d \nu(z) \to 0 \ \ \text{as} \ \  k\to\ii.
$$
Therefore
\begin{equation*}
\begin{split}
\int_{\mbb C_1}\|\wt g(\lb^{s_n}z)&-\wt g(z)\|_1\,d \nu(z)
\leq\int_{\mbb C_1}\|\wt g(\lb^{s_n}z)-\wt h_k(\lb^{s_n}z)\|_1\,d \nu(z)\\
&+\int_{\mbb C_1}\|\wt h_k(\lb^{s_n}z)-\wt h_k(z)\|_1\,d \nu(z)+
\int_{\mbb C_1}\|\wt h_k(z)-\wt g(z)\|_1\,d \nu(z)\\
&=2\int_{\mbb C_1}\|\wt h_k(z)-\wt g(z)\|_1\,d \nu(z)+\int_{\mbb C_1}\|\wt h_k(\lb^{s_n}z)-\wt h_k(z)\|_1\,d \nu(z)
\end{split}
\end{equation*}
implies that (\ref{e221}) holds for any simple Bochner measurable function $\wt g$.

As $\wt f\in L^1(\wt{\mc M},\wt\tau)$, there exists a sequence $\{\wt g_k(z)\}_{k=1}^\ii$ of simple Bochner measurable  functions for which
$$
\int_{\mbb C_1}\|\wt f(z)-\wt g_k(z)\|_1\,d\nu(z) \to 0\text{ \ \ as\ \ } k\to \ii.
$$
Then, repeating the previous argument, we conclude that the convergence in (\ref{e221}) holds for $\wt f$. Therefore, it now follows that
$$
\|T_{s_n}^{(\lb)}(\wt f)-\wt f\|_{L^1(\wt{\mc M},\wt\tau)}\to 0 \text{ \ \ as\ \ } n \to \ii.
$$

Finally, let $t \geq 0$, $t_n > 0$, $t_n \downarrow t$ and denote $s_n=t_n - t$. Then we have
\begin{equation*}
\begin{split}
\|T_{t_n}^{(\lb)}(\wt f)-T_{t}^{(\lb)}(\wt f)\|_{L^1(\wt{\mc M},\wt\tau)}&=
\|T_{t+s_n}^{(\lb)}(\wt f)-T_{t}^{(\lb)}(\wt f)\|_{L^1(\wt{\mc M},\wt\tau)}\\
&\leq\|T_{s_n}^{(\lb)}(\wt f)-\wt f\|_{L^1(\wt{\mc M},\wt\tau)} \to 0
\end{split}
\end{equation*}
as $n\to\ii$. The case $t_n \uparrow t>0$ is similar.
\end{proof}

\begin{lm}\label{l22}
Let $\{T_t\}_{t\ge 0}\su DS^+$ be a strongly continuous semigroup in $L^1$, and let $p(t)=\sum\limits_{j=1}^nw_j\lb_j^t$ be a trigonometric polynomial. If $x\in L^1$ and
$$
P_t(x)=\frac 1t\int_0^tp(s)T_s(x)ds,
$$
then
\begin{enumerate}[(i)]
\item
the averages $P_t(x)$ converge a.u. as $t\to \ii$;
\item
the averages $P_t(x)$ converge a.u. to $p(0)x$ as $t\to 0$.
\end{enumerate}
\end{lm}
\begin{proof}

(i) Fix  $\lb \in \mbb C_1$ and let
$$
T_t^{(\lb)}(\wt f)(z)=T_t(\wt f(\lb^tz)),\ \ \ t\ge 0,\ \ \wt f\in L^1(\wt{\mc M},\wt\tau)+\wt{\mc M}, \ z\in \mbb C_1.
$$
In view of Proposition \ref{p10}, $\{T_t^{(\lb)}\}_{t\ge 0}\su DS^+(\wt{\mc M},\wt\tau)$ is a strongly continuous semigroup on $L^1(\wt M,\wt \tau)$. Then, by \cite[Corollary 5.2]{cl1},
given $\wt f\in  L^1(\wt{\mc M},\wt\tau)$, the averages
\begin{equation}\label{e0}
\frac 1t\int_0^tT_s^{(\lb)}(\wt f)ds
\end{equation}
converge a.u. as $t\to \ii$. Therefore, by Lemma \ref{l21}, the averages
$$
\frac 1t\int_0^tT_s^{(\lb)}(\wt f)(z)ds=\frac 1t\int_0^tT_s(\wt f(\lb^sz))ds
$$
converge a.u. as $t\to \ii$ for $\nu$-almost all $z\in \mbb C_1$. In particular, letting $\wt f(z)=zx$, we conclude
that the averages
$$
z\,\frac 1t\int_0^t\lb^sT_s(x)ds
$$
converge a.u. as $t\to\ii$ for some $0\neq z\in \mbb C_1$, implying that the averages
$$
\frac 1t\int_0^t\lb^sT_s(x)ds
$$
converge a.u. as $t\to\ii$. Therefore, by linearity, the averages $P_t(x)$ converge a.u. as $t\to\ii$.

(ii) Now, by \cite[Theorem 5.1]{cl2}, if $\wt f\in  L^1(\wt{\mc M},\wt\tau)$, it follows that the averages (\ref{e0})
converge a.u. to $\wt f$ as $t\to 0$. Then, letting $\wt f(z)=zx$, we see as above that
$$
\frac 1t\int_0^t\lb^sT_s(x)ds\to x \text{ \ a.u.}
$$
as $t\to 0$, and the result follows by linearity.
\end{proof}

Here is a Besicovitch-weighted noncommutative individual ergodic theorem for flows in $L^1$ generated by $L^1$-strongly continuous semigroups of positive Dunford-Schwartz operators:
\begin{teo}\label{t21}
Let $\{T_t\}_{t\ge 0}\su DS^+$ be a strongly continuous semigroup in $L^1$, and let $\bt(t)$ be a bounded Besicovitch function with $\|\bt\|_\ii<C<\ii$. Then, given $x\in L^1$, the averages (\ref{e22}) converge a.u. to some $\wh x\in L^1$ as $t\to \ii$.
\end{teo}
\begin{proof}
Assume first that $x\in L^1\cap \mc M$. Fix $\ep>0$ and choose a trigonometric polynomial $p=p_\ep$ to satisfy
condition (\ref{e21}). Let $\{P_t(x)\}_{t>0}$ be the corresponding averages from Lemma \ref{l22}. Then we have
$$
\| B_t(x)-P_t(x)\|_\ii\leq\frac 1t\int_0^t|\bt(s)-p(s)|ds\, \|x\|_\ii<\ep\,\|x\|_\ii
$$
for all big enough values of $t$. Since, by Lemma \ref{l22}, the averages $P_t(x)$ converge a.u.,
it follows that the net $\{B_t(x)\}_{t>0}$ is a.u. Cauchy as $t\to\ii$.

Now, let $x\in L^1$. Without loss of generality, assume that $x\in L^1_+$, and let $\{e_\lb\}$ be the spectral family
of $x$. Given $m\in \mbb N$, if we define $y_m=\int_0^m\lb de_\lb$ and $x_m=x-y_m$, then
$\{y_m\}\su L^1_+\cap \mc M$, $\{x_m\}\su L^1_+$ and $\|x_m\|_1\to 0$.

Fix $\ep>0$ and $\dt>0$. If $\{t_n\}$ is a sequence of positive rational numbers which is dense in $(0,\ii)$, then, by Lemma \ref{l23}, the sequence $\{B_{t_n}\}$ is b.u.e.m. at zero on $L^1_+$, hence on $L^1$ (see \cite[Lemma 4.1]{li}). Applying Lemma \ref{l24}, we find a projection $e\in \mc P(\mc M)$ and $x_{m_0}\in \{x_m\}$ such that
$$
\tau(e^\perp)\leq \frac \ep2 \text{ \ \ and \ \ } \sup_n\|B_{t_n}(x_{m_0})e\|_\ii \leq\frac \dt3.
$$

If $t>0$, then $t_{n_k}\to t$ for a subsequence $\{t_{n_k}\}$, and it easily verified that
$$
\|B_t(x_{m_0})-B_{t_{n_k}}(x_{m_0})\|_1\to 0.
$$
Therefore $B_{t_{n_k}}(x_{m_0})\to B_t(x_{m_0})$ in measure, which implies that there is a subsequence
$\{t_{n_{k_l}}\}$ such that $B_{t_{n_{k_l}}}(x_{m_0})\to B_t(x_{m_0})$ a.u.

Since $\|B_{t_{n_{k_l}}}(x_{m_0})e\|_\ii<\dt/3$ for each $l$, it follows from \cite[Lemma 5.1]{cl1} that
\begin{equation}\label{eq2}
\sup_{t>0}\|B_t(x_{m_0})e\|_\ii \leq \frac \dt 3.
\end{equation}
Because $y_{m_0}\in L^1 \cap \mc M$, the net $B_{t}(y_{m_0})$ is a.u. Cauchy as $t\to \ii$.
Therefore, there exist $g \in \mc P(\mc M)$ and $t_0 > 0$ such that
\begin{equation}\label{eq3}
\tau(g^\perp)\leq \frac\ep 2 \text { \ \ and \ \ }
\|(B_t( y_{m_0})- B_{t'}( y_{m_0}))g\|_\ii \leq \frac \dt3.
\end{equation}
for all $t, t' \ge t_0$.

If $h = e \wedge g$, then $\tau(h^\perp)\leq \ep $ and, in view of (\ref{eq2})  and (\ref{eq3}), we have
\begin{equation*}
\begin{split}
\| (B_t(x)- B_{t'}(x))h\|_\ii &\leq \| (B_t(y_{m_0})- B_{t'}( y_{m_0}))h\|_\ii \\
&+\| B_t(x_{m_0})h\|_\ii + \| B_{t'}( x_{m_0})h\|_\ii \leq \delta
\end{split}
\end{equation*}
for all $t, t' \ge t_0$.
Thus, the net $\{B_{t}(x)\}$ is a.u. Cauchy as $t \to \ii$. Since $L^0$ is complete with respect to a.u.
convergence (see proof of \cite[Theorem 2.3]{cls}), there is $\wh x\in L^0$ such that $B_t(x)\to \wh x$ a.u.
In particular, $B_t(x)\to \wh x$ in measure. Since $|\bt(t)|<C<\ii$ for all $t\ge 0$, each map $C^{-1}B_t$ is a contraction in $L^1$, which, because the unit ball of $L^1$ is closed in measure topology, implies that
$\wh x\in L^1$.
\end{proof}

The following theorem is a local ergodic theorem in $L^1$ for $L^1$-continuous semigroups of positive Dunford-Schwartz operators modulated by bounded zero-Besicovitch functions.

\begin{teo}\label{t22}
Let $\{T_t\}_{t\ge 0}\su DS^+$ be a strongly continuous semigroup in $L^1$, and let $\bt(t)$  be a zero-bounded Besicovitch function with $\|\bt\|_\ii<C<\ii$. Then, given $x\in L^1$, the averages (\ref{e22})
converge a.u. to $\alpha(x)\,x$ as $t\to 0$ for some $\alpha(x) \in \mbb C$.
\end{teo}
\begin{proof}
Assume first that $x\in L^1\cap\mc M$. Fix $\ep>0$ and choose a trigonometric polynomial $p=p_\ep$ to satisfy condition (\ref{e21z}). If $\{P_t(x)\}_{t>0}$ are the  averages from Lemma \ref{l22}, then
$$
\| B_t(x)-P_t(x)\|_\ii\leq\frac 1t\int_0^t|\bt(s)-p(s)|ds\, \|x\|_\ii<\ep\,\|x\|_\ii
$$
for all small  enough values of $t$. Since, by Lemma \ref{l22}, the averages $P_t(x)$ converge a.u. as $t\to 0$,
it follows that the net $\{B_t(x)\}_{t>0}$ is a.u. Cauchy as $t\to 0$. Then, repeating the proof of the Theorem \ref{t21}, we obtain that for any $x \in L^1$ the averages $\{B_t(x)\}_{t>0}$ converge a.u. to some $\wh x\in L^1$ as $t\to 0$.

Let $p_n$ be a trigonometric polynomial to satisfy condition (\ref{e21z}) with $\ep=1/n$. If $\{P^{(n)}_t(x)\}_{t>0}$ are the corresponding averages from Lemma \ref{l22}, then there is $t_n>0$ such that
\begin{equation}\label{eq4}
\| B_t(x)-P^{(n)}_t(x)\|_1\leq\frac 1t\int_0^t|\bt(s)-p(s)|ds\, \|x\|_1<\frac{\|x\|_1}n
\end{equation}
for all  $0<t<t_n$.

By Lemma \ref{l22}, $P^{(n)}_t(x)\to p_n(0)\,x$ a.u., hence $B_t(x)-P^{(n)}_t(x)\to \wh x-p_n(0)\,x$ a.u., implying that
$$
B_t(x) -P^{(n)}_t(x)\to\wh x-p_n(0)\,x \text{ \ \ in measure as\ \ } t\to 0.
$$
Since the unit ball of $L^1$ is closed in measure topology, (\ref{eq4}) entails that
$$
\|\wh x - p_n(0)\,x\|_1\leq\frac{\|x\|_1}n,
$$
hence $\wh x=\|\cdot\|_1-\lim\limits_{n\to\ii}p_n(0)\,x$, and we conclude that $\wh x=\al(x)\,x$ for some $\al(x)\in\mbb C$.

Now, let $0\leq x\in L^1$, $x\neq 0$, $e_n=\{x\leq n\}$, $n\in\mbb N$, and let $x_n= x\,e_n$.  It is clear that
$\{x_n\}\su L^1\cap\mc M$ and $\|x-x_n\|_1 \to 0$ as $n\to\ii$.  As shown above, $B_t(x)\to\wh x$ a.u. as $t\to 0$ for some
$\wh x\in L^1$ and $B_t(x_n)\to\al(x_n)\,x_n$ as $t\to 0$ for for every $n$ and some $\al(x_n)\in\mbb C$. Consequently,
$B_t(x)-B_t(x_n)\to\wh x-\al(x_n)\, x_n$ in measure.

Besides,
$$
\|B_t(x)-B_t(x_n)\|_1\leq\frac 1t\int_0^t|\bt(s)|ds\, \|x-x_n\|_1 \leq C\,\|x-x_n\|_1.
$$
Since the unit ball of $L^1$ is closed in measure topology, it follows  that
\begin{equation}\label{eq41}
\|\wh x-\al(x_n)\,x_n\|_1\leq C\,\|x-x_n\|_1 \to 0 \ \ \text{as} \ n \to \infty.
\end{equation}
Choose $k \in \mbb N$ such that $x_k=x\,e_k\neq 0$. If $n > k$, then we have
$$
\|\wh x\,e_k-\al(x_n)\,x_k\|_1=\|\wh x\,e_k-\al(x_n)\,x_n\,e_k\|_1\leq\|\wh x-\al(x_n)\,x_n\|_1,
$$
so, (\ref{eq41}) implies that
$$
\wh x\,e_k=\|\cdot\|_1-\lim\limits_{n\to\ii}\al(x_n)\,x_k.
$$
Therefore, there exists $\lim\limits_{n\to\ii}\al(x_n)=\al(x)$, implying that $\|\al(x_n)\,x_n-\al(x)\,x\|_1\to 0$ as $n\to\ii$, hence $\wh x=\al(x)\,x$, in view of (\ref{eq41}).

If $x\in L^1$, we employ the decomposition $x=x_1-x_2+i(x_3-x_4)$, $0\leq x_i\in L^1$, $i=1,2,3,4$, and apply the above argument to each $x_i$.
\end{proof}

A weaker version of Theorem \ref{t22} - for b.a.u. convergence and with no identification of the limit - was announced in \cite{mk}. However, some steps leading to the main Theorem 3.1 of the paper, such as the fact given above in Proposition \ref{p10}, were left unjustified.

\section{Extension to noncommutative fully symmetric spaces}
Now we will extend the results of Section 3 to the noncommutative fully symmetric spaces $E\su L^1+L^\ii$ with
$\mb 1\notin E$, in particular, to the spaces $L^p$, $1<p<\ii$.

Let $x\in L^0$, and let $\{ e_{\lb}\}_{\lb\ge 0}$ be the spectral family of projections for the absolute value
$| x|=(x^*x)^{1/2}$ of $x$. If $t>0$, then the {\it $t$-th generalized singular number} of $x$ (the {\it non-increasing rearrangement} of $x$) is defined as
$$
\mu_t(x)=\inf\{\lb>0:\ \tau(e_{\lb}^{\perp})\leq t\}
$$
(see \cite{fk}).

Let $\mu$ be the Lebesgue measure on $(0,\ii)$. It is well known that
$$
L^p=L^p(\mc M,\tau) =\Big\{x\in L^0: \int_0^\ii\mu_t(x)^p d\mu(t) < \ii\Big\},
$$
and $\|x\|_p = \|\mu_t(x)\|_p$, $x\in L^p$, $1\leq p<\ii$ (see, for example, \cite{fk}).

Let $(E,\|\cdot\|_E)$ be a symmetric function space on $((0,\ii),\mu)$ (see, for example, \cite[Ch.II, \S\,4]{kps}). Define
$$
E(\mc M)=E(\mc M, \tau)=\{ x\in L^0:\ \mu_t(x)\in E\}
$$
and set
$$
\| x\|_{E(\mc M)}=\| \mu_t(x)\|_E, \ x\in E(\mc M).
$$
It is shown in \cite{ks} that $(E(\mc M), \| \cdot \|_{E(\mc M)})$ is a Banach space and that conditions
$$
x\in E(\mc M), \ y\in L^0, \ \mu_t(y)\leq \mu_t(x) \text{ \ for all \ }t>0
$$
imply that $y\in E$ and $\| y\|_E\leq \| x\|_E$, in which case $(E(\mc M), \| \cdot \|_{E(\mc M)})$ is said to be a
{\it noncommutative  symmetric space}.

A noncommutative  symmetric space $(E(\mc M), \| \cdot \|_{E(\mc M)})$ is called {\it fully symmetric} if conditions
$$
x\in E(\mc M), \ y\in L^0, \ \int_0^s\mu_t(y)dt\leq  \int_0^s\mu_t(x)dt \ \ \forall \ s>0 \ \  (\text {writing } \  y \prec\prec x)
$$
imply that $y\in E$ and $\| y\|_E\leq \| x\|_E$. For example, $L^p(\mc M)=L^p$, $1\leq p\leq\ii$,
and the  Banach spaces
$$
 (L^1\cap L^\ii)(\mc M) = L^1\cap\mc M \text{ \ \ with \ \ } \|x\|_{L^1\cap \mc M}=\max\{ \|x\|_1, \|x\|_\ii\}
\text{ \ \ and \ \ }
$$
$$
(L^1 + L^\ii)(\mc M) = L^1+\mc M \text{ \ \ with \ \ }
$$
$$
\|x\|_{L^1+\mc M}=\inf \left \{ \|y\|_1+ \|z\|_{\infty}: \ x = y + z, \ y \in L^1, \ z \in \mc M \right \}= \int_0^1 \mu_t(x) dt
$$
are noncommutative fully symmetric spaces (see \cite{ddp}).

Since, given a symmetric function space $E=E(0,\ii)$,
$$
L^1(0,\ii) \cap L^\ii(0,\ii) \su E \su L^1(0,\ii)+L^\ii(0,\ii),
$$
with continuous embedding \cite[Ch.II, \S\,4, Theorem 4.1]{kps}, it follows that
$$
L^1(\mc M) \cap \mc M \subset E(\mc M) \subset L^1(\mc M) + \mc M,
$$
with continuous embedding.

Define
$$
\mc R_\tau= \{x \in L^1 + \mc M: \ \mu_t(x) \to 0 \text{ \ as \ } t\to \ii\}.
$$
It is known that $\mc R_\tau$ is the closure of $L^1\cap\mc M$ in $L^1+\mc M$ \cite[Proposition 2.7]{ddp}, in particular, $(\mc R_\tau, \|\cdot\|_{L^1+\mc M})$ is a noncommutative  fully symmetric space \cite{cl1}. In addition, if $\tau(\mb 1)=\ii$, then a symmetric space $E(\mc M, \tau)$ is contained in
$\mc R_\tau$ if and only if $\mb 1\notin E(\mc M, \tau)$ \cite[Proposition 2.2]{cl1}.

Every noncommutative fully symmetric space $E=E(\mc M)$ is an {\it exact interpolation space} for the Banach couple $(L^1(\mc M),\mc M)$  \cite{ddp1}. Therefore $T(E)\su E$ and $\| T\|_{E\to E}\leq 1$ for $T\in DS$; in particular,
$$
T(\mc R_\tau)\su \mc R_\tau \text{ \ \ and \ \ } \| T\|_{{L^1+\mc M}\to{L^1+\mc M}} \leq 1.
$$

Let $\{T_t\}_{t\ge 0}\su DS^+$ be a strongly continuous semigroup in $L^1$, $\bt: \mbb R_+\to \mbb C$ a Lebesgue measurable function with $\|\bt\|_\ii<C<\ii$, and let $B_t(x)$, $x\in L^1$, $t>0$, be given by (\ref{e22}). We have  $\bt(t)=\bt_1(t)-\bt_2(t)+i\,(\bt_3(t)-\bt_4(t))$, where $\bt_j:\mathbb R_+\to\mathbb R_+$ is Lebesgue measurable function such that $\|\bt_j \|_\ii<C<\ii$ for each $j=1,\dots,4$. Denote
$$
B_t^{(j)}(x)=\frac1t\int_0^t\bt_j(s)T_s(x)ds, \ \ x\in L^1, \ t>0.
$$
Then for each $j$,
$$
\| B_t^{(j)}(x)\|_1\leq C\| x\|_1 \ \ \forall \ x\in L^1 \text{ \ and \ } \| B_t^{(j)}(x)\|_\ii\leq C \| x\|_\ii \ \ \forall \ x\in L^1\cap L^\ii.
$$
Consequently, $C^{-1}B_t^{(j)}$ is a positive {\it absolute contraction} in $L^1$, which, by \cite[Proposition 1.1]{cl}, admits a unique extension to a positive Dunford-Schwartz operator $D_t^{(j)}$. Therefore, $C D_t^{(j)}$, $t>0$, is the unique extension of $B_t^{(j)}$ to the Banach space $L^1+ L^\ii$. By linearity, $B_t$ admits a unique extension to $L^1+ L^\ii$, which we will also denote by $B_t$.

Since a noncommutative fully symmetric space $E=E(\mc M)$ is an exact interpolation space for the Banach couple $(L^1,\mathcal M)$, it now follows that $B_t(E)\su E$ and $\|B_t\|_{E\to E}\leq C$ for every $t>0$.

\begin{teo}\label{t41}
Let $\{T_t\}_{t\ge 0}\su DS^+$ be a strongly continuous semigroup in $L^1$, and let $\bt(t)$ be a bounded Besicovitch function such that $\|\bt(t)\|_\ii<C$. Then, given $x\in \mc R_\tau$, the averages (\ref{e22})
converge a.u. to some $\wh x\in \mc R_\tau$ as $t\to\ii$.
\end{teo}
\begin{proof}
Without loss of generality assume that $x\ge 0$, and let $\{e_\lb\}_{\lb\ge 0}$ be the spectral family of $x$. Given $m\in\mbb N$, denote $x_m=\int_{1/m}^\ii\lb de_\lb$ and $y_m = \int_0^{1/m} \lb de_\lb$. Then we have
$0 \leq y_m\leq m^{-1}\,\mathbf 1$, $x_m \in L^1$, and $x = x_m + y_m$ for all $m$.

Fix $\ep > 0$. By Theorem \ref{t21}, $B_t(x_m) \to \wh  x_m \in L^1$ a.u. as $t \to \ii$ for each $m$.
Therefore, there exists a sequence $\{e_m\}\su \mc P(\mc M)$ such that
$$
\tau(e_m^{\perp})\leq \frac \ep{2^m} \text{ \ \ and \ \ }
\| (B_t(x_m)- \wh x_m)e_m\|_\ii\to 0 \text{ \ \ as \ \ } t \to \ii.
$$
Then it follows that, for for some $t(m)>0$,
$$
\|(B_t(x_m)-B_{t'}(x_m))e_m\|_{\ii} < \frac1m \ \ \ \forall \ t, t' \ge t(m).
$$
Since $\|y_m\|_\ii\leq m^{-1}$, we have
\begin{equation*}
\begin{split}
\| (B_t(x)-B_{t'}(x))e_m\|_\ii &\leq
\| ((B_t(x_m)-B_{t'}(x_m))e_m\|_{\ii} + \| (B_t(y_m)-B_{t'}(y_m))e_m\|_\ii\\
&<\frac1m + \|B_t(y_m)e_m\|_{\ii} + \| B_{t'}(y_m)e_m\|_{\ii}  \leq \frac{1+2C}m
\end{split}
\end{equation*}
for each $m$ and all $t, \ t' \geq t(m)$.

If $e =\bigwedge\limits_{m\in\mbb N} e_m$, then
$$
\tau(e^\perp)\leq\ep\text{ \ \ and \ \ } \| (B_t(x)-B_{t'}(x))e\|_\ii< \frac{1+2C}m \ \ \ \forall \ t, t'\ge t(m).
$$
This means that $\{B_t(x)\}_{t>0}$ is a Cauchy net with respect to a.u. convergence.
Since $L^0$ is complete with respect to a.u. convergence \cite[Remark 2.4]{cls}, we conclude that the net
$\{B_t(x)\}_{t>0}$ converges a.u. to some $\widehat x \in L^0$.

As $B_t(x)\to\wh x$ a.u., it is clear that $B_t(x)\to\wh x$ in measure.
Since  $L^1+\mc M$ satisfies the Fatou property \cite[\S 4]{dp}, its unit ball is closed in measure topology \cite[Theorem 4.1]{ddst},  and the inequality $\|B_t(x)\|_{L^1+\mc M}\leq C\,\|x\|_{L^1+\mc M}$ implies that $\widehat x \in L^1+\mc M$.
In addition, $B_t(x)\to\widehat x$ in measure implies that $\mu_s(B_t(x))\to \mu_s(\widehat x)$ almost everywhere on
$((0,\ii),\mu)$ for each $s>0$ (this can be shown as in the commutative case; see, for example, \cite[Ch.II, \S\,2, Property $11^\circ$]{kps}).

Since  $C^{-1}B_t \in DS^+$, we have $C^{-1}\,\mu_s(B_t(x))=\mu_s(C^{-1}B_t(x))\prec\prec \mu_t( x)$. This means  that
$$
C^{-1}\int_0^u\mu_s(B_t(x))ds =\int_0^u\mu_s(C^{-1}B_t(x))ds\leq\int_0^u\mu_s(x)ds \ \ \ \forall \  u>0,\ t > 0.
$$
Now, Fatou property for $L^1((0,u), \mu)$ and $\mu_s(B_t(x))\to \mu_s(\wh x)$ in measure on $((0,u), \mu)$ imply that
$$
\int_0^u\mu_s(\widehat x)ds\leq C\,\int_0^u\mu_s(x)ds=\int_0^u\mu_s(C\,x)ds
\ \ \ \forall \ u>0,
$$
that is, $\mu_s(\widehat x)\prec\prec \mu_s(C\,x)$. Therefore, since $\mc R_\tau$ is a fully symmetric space and
$C\,x\in\mc R_\tau$, it follows that $\wh x\in\mc R_\tau$.
\end{proof}

The following is an application of Theorem \ref{t41} to fully symmetric spaces.

\begin{teo}\label{t42}
Let $E=E(\mc M,\tau)$ be a noncommutative fully symmetric space such that $\mb 1\notin E$, and let $\{T_t\}_{t\ge 0}$ and
$\bt(t)$ be as in Theorem \ref{t41}. Then, given $x\in E$, the averages (\ref{e22}) converge a.u. to some $\wh x\in E$ as
$t\to\ii$.
\end{teo}
\begin{proof}
Since $\mb 1\notin E$, it follows that $E\su\mc R_\tau$. Thus, by Theorem \ref{t41}, the averages $\{B_t(x)\}_{t>0}$
 converge a.u. to some $\wh x\in\mc R_\tau$ as $t \to \ii$.  Since $\mu_s(\wh x)\prec\prec\mu_s(C\,x)$
(see the proof of Theorem \ref{t41}) and $E$ is a fully symmetric space, we conclude that $\wh x\in  E$.
\end{proof}

Repeating the proofs of Theorems \ref{t41} and \ref{t42}, we obtain the following extension of Theorem \ref{t22}.
\begin{teo}\label{t43}
Let $\{T_t\}_{t\ge 0}\su DS^+$ be a strongly continuous semigroup in $L^1$, $\bt(t)$ be a bounded zero-Besicovitch function, and let  $E=E(\mc M,\tau)$ be a noncommutative fully symmetric space such that $\mb 1\notin E$. Then,  given
$x\in E$, the averages (\ref{e22}) converge a.u. to $\al(x)\,x$ as $t\to 0$ for some $\al(x)\in\mbb C$.
\end{teo}

 \begin{rem}
Employing the approach in the proof of \cite[Theorem 5.7]{cl1} (see also \cite{cl2}), one can see that the assertion of Theorem \ref{t42} (Theorem \ref{t43}) remains valid when the semigroup $\mbb R_+$ is replaced by $\mbb R_+^d$,
$d\in \mbb N$, so that the averages
$$
{\mb B}_t(x)=\frac1{t^d}\int_{[0,t]^d}\bt(\mb s)T_{\mb s}(x)d\mb s, \ \ x\in E,
$$
converge a.u. to some $\wh x\in E$ as $t\to\ii$ (respectively, to $\al(x)\,x$ as $t\to 0$).
\end{rem} 

\section{Applications to noncommutative Orlicz, Lorentz and Marcinkiewicz spaces}

Below we give applications of Theorems \ref{t42} and \ref{t43} to noncommutative Orlicz, Lorentz and Marcinkiewicz spaces.

1. Let $\Phi$ be an {\it Orlicz function}, that is, $\Phi:[0,\ii)\to [0,\ii)$ is left-continuous, convex, increasing and such that $\Phi(0)=0$ and $\Phi(u)>0$ for some $u\ne 0$ (see, for example \cite[Ch.2, \S 2.1]{es}).
If $0 \leq x\in L^0$ and $x=\int_0^\ii\lb de_\lb$ its spectral decomposition, one can define $\Phi(x)=\int_0^\ii\Phi(\lb) de_\lb\in L^0$.

The noncommutative Orlicz space associated with $(\mc M,\tau)$ is the set
$$
L^\Phi=L^\Phi(\mc M, \tau)=\Big\{ x \in L^0:  \Phi(a^{-1}\,|x|)\in  L^1 \text { \ for some \ } a>0\Big\}.
$$
The {\it Luxemburg norm} of an operator  $x \in L^{\Phi}$ is defined as
$$
\| x\|_\Phi=\inf\left \{ a>0: \|\Phi(a^{-1}\,|x|)\|_1 \leq 1\right\}.
$$
It is known that $(L^\Phi, \| \cdot\|_\Phi)$ is a noncommutative fully symmetric space (see, for example, \cite[Corollary 2.2]{cl3}). In addition,
$$
L^\Phi=\left\{x \in L^0: \ \mu_t(x) \in L^\Phi(0,\ii)\right\}\text{ \ \ and \ \ }\| x \|_\Phi=\| \mu_t(x) \|_\Phi\ \ \forall \ x \in L^\Phi.
$$
\cite[Corollary 2.1]{cl3}). The pair $(L^\Phi, \| \cdot\|_\Phi$)  is called the {\it noncommutative Orlicz space} (see, for example \cite{ku}).

If $\tau(\mb 1)<\ii$, then $L^\Phi\su L^1$. If $\tau(\mb 1)=\ii$ and $\Phi(u)>0$ for all $u\ne 0$, then
$ \Phi(a^{-1}\,\mb 1)\notin  L^1$ for each $a>0$, hence $\mb 1\notin L^\Phi$. Therefore, Theorems \ref{t42} and \ref{t43} imply the following.
\begin{teo} \label{t51}
Let $\Phi$ be an Orlicz function such that $\Phi(u)>0$ for all $u>0$.
Let  $\{T_t\}_{t\ge 0}\su DS^+$ be a strongly continuous semigroup in $L^1$, and let $\bt(t)$ be a bounded  Besicovitch (zero-Besicovitch) function. Then, given $x\in L^\Phi$, the averages (\ref{e22}) converge a.u. to some
$\wh x\in L^\Phi$ as $t\to\ii$ (respectively, to $\al(x)\,x$ as $t\to 0$ for same $\al(x)\in\mbb C$).
\end{teo}

2. Let $\mu$ be the Lebesgue measure on the interval $(0,\ii)$. Let $L^0(0,\ii)$ be the algebra of (equivalence classes of) almost everywhere finite real-valued measurable functions $f$ on the measure space $((0,\ii),\mu)$ with $\mu\{|f|>\lb\}<\ii$ for some $\lb>0$. The {\it non-increasing rearrangement} of a function $f \in L^0(0,\ii)$ is the function $f^*$ on $(0,\ii)$ defined by
\begin{equation}\label{re}
f^*(t)=\inf\{\lb>0:\ \mu\{|f|>\lb\}\leq t\}
\end{equation}  
(see, for example, \cite[Ch.II, \S\,2]{kps}). Let $\psi$ be a concave function on $[0, \infty)$ with $\psi(0) = 0$ and
$\psi(t) > 0$ for all $t > 0$, and let
$$
\Lambda_\psi(0,\ii)=\left\{f \in  L^0(0,\ii): \ \|f \|_{\Lb_\psi}=\int_0^\ii f^*(t)d\psi(t)<\ii\right\}
$$
be the corresponding {\it Lorentz space}.

It is well known that $(\Lambda_\psi(0,\ii), \|\cdot\|_{\Lambda_\psi})$
is a fully symmetric function space; in addition, if $\psi(\ii)=\lim\limits_{t\to\ii}\psi(t)=\ii$, then $\mb1\notin\Lb_\varphi(0,\ii)$ \cite[Ch.II, \S\,5]{kps}).

Define the {\it noncommutative Lorentz space} (see, for example, \cite{cms}) as
$$
\Lb_\psi=\Lb_\psi(\mc M,\tau)=\left\{ x\in L^0: \ \mu_t(x)\in\Lb_\psi(0,\ii)\right\}
$$
and set
$$
\| x\|_{\Lb_\psi}=\|\mu_t(x)\|_{\Lb_\psi},  \ x\in\Lb_\psi.
$$
It is known \cite{ks} that $(\Lb_\psi,\| \cdot\|_{\Lb_\psi})$  is a noncommutative fully symmetric space. If $\tau(\mb 1)<\ii$, then $\Lb_\psi\su L^1$. If $\tau(\mb 1)=\ii$ and $\psi(\ii)=\ii$, then it follows that $\mb1\notin\Lb_\psi$. Therefore, Theorems \ref{t42} and \ref{t43} imply the following.
\begin{teo} \label{t52}
Let $\psi$ be a concave function on $[0,\ii)$ with $\psi(0) = 0$ and
$\psi(t)>0$ for all $t > 0$, and let $\psi(\ii)=\ii$. Let $\{T_t\}_{t\ge 0}$ and $\bt(t)$ be as in Theorem \ref{t51}.
Then, given $x\in \Lambda_\psi$, the averages (\ref{e22}) converge a.u. to some $\wh x\in\Lb_\psi(\mc M, \tau)$ as $t\to \ii$ (respectively, to $\al(x)\,x$ as $t\to 0$ for same $\al(x)\in\mbb C$).
\end{teo}

3. Let $\psi$ be as above, and let
$$
M_\psi(0,\ii)=\left\{f\in L^0(0,\ii): \ \|f \|_{M_\psi}=\sup\limits_{0<s<\ii}\frac1{\psi(s)}\int_0^s f^*(t) d t<\ii\right\}
$$
be the corresponding {\it Marcinkiewicz space}. It is known that $(M_\psi(0,\ii), \|\cdot\|_{M_\psi})$ is a fully symmetric function space. Besides,  $\mb 1\notin  M_\psi(0,\ii)$ if and only if $\lim\limits_{t\to\ii}\frac{\psi(t)}{t}=0$ \cite[Ch.II, \S\,5]{kps}). Also, if $\psi(+0)>0$ and $\psi(\ii)<\ii$, then $M_\psi(0,\ii)= L^1(0,\ii)$ as sets.

Define the {\it noncommutative  Marcinkiewicz space} as
$$
M_\psi=M_\psi(\mc M, \tau)=\left\{ x\in L^0: \ \mu_t(x)\in M_\psi(0,\ii)\right\}
$$
and set
$$
\| x\|_{M_\psi}=\| \mu_t(x)\|_{M_\psi},  \ x\in M_\psi(\mc M, \tau)
$$
(see, for example, \cite{cms}). It is known \cite{ks} that $(M_\psi,\|\cdot\|_{M_\psi})$  is a noncommutative fully symmetric space. If $\tau(\mb 1)<\ii$ or $\psi(+0)>0$ and $\psi(\ii)<\ii$, then $M_\psi\su L^1$. If  $\tau(\mb 1)=\infty$ and
$\lim\limits_{t\to\ii}\frac{\psi(t)}{t}=0$, then  $\mb 1\notin M_\psi$.

Thus, the corresponding version of Theorem \ref{t52} holds for Marcinkiewicz spaces $M_\psi(\mc M, \tau)$ if we replace condition $\psi(\ii)=\ii$  by $\lim\limits_{t\to\ii}\frac{\psi(t)}t=0$.

\section{The commutative case}

In this section we present applications of Theorems \ref{t42} and \ref{t43} to fully symmetric function spaces.

Let $(\Om,\nu)=(\Om,\mc A,\nu)$ be a {\it Maharam measure space} (see, for example, \cite[Ch.1, \S\,1.2, Sections 1.1.7, 1.1.8]{Ku}), that is,
\begin{enumerate}[(i)]
\item
$\nu$ is a countably additive function defined on a $\sigma$-algebra $\mc A$ of subsets of a set $\Om$ with the values in the extended half-line $[0,\ii]$;
\item
if $A\su\Om$ and $A\cap F\in\mc A$ for all $F\in\mc A$ with $\nu(F)<\ii$, then $A\in\mc A$;
\item
if $A\su F\in \mc A$ and $\nu(F)$=0, then $A\in\mc A$;
\item
for any $A\in\mc A$ there exists $F\in\mc A$ such that $F\su A$ and $\nu(F)<\ii$;
\item
the Boolean algebra $\wh{\mc A}$ of classes of $\nu$-almost everywhere equal sets in $\mc A$ is order complete.
\end{enumerate}
Clearly, every complete $\sigma$-finite measure space is Maharam measure space.

It is well known that the $\ast$-algebra $L^\ii(\Om,\nu)$ of (equivalence classes of) essentially bounded measurable complex-valued functions defined on a Maharam measure space $(\Om,\nu)$ is a commutative von Neumann algebra with a semifinite normal faithful trace $\nu(f)=\int_\Om f\,d\nu$, $0\leq f \in L^\ii(\Om,\nu)$.
The converse is also true: any commutative von Neumann algebra $\mc M$  with a semi-finite normal faithful trace
$\tau$ is $\ast$-isomorphic to the $\ast$-algebra $L^\ii(\Om,\nu)$ for some Maharam measure space $(\Om,\nu)$ such that
$\tau(f)=\int_\Om f\,d\nu$, $0 \leq f \in L^\ii(\Om,\nu)$ (see, for example, \cite[Ch.7, \S\,7.3]{di}).

If $(\Om,\nu)$ is a Maharam measure space and $\mc M=L^\ii(\Om,\nu)$ with $\tau(f)=\int_\Om f\,d\nu$, $0\leq f\in L^\ii(\Om,\nu)$, then $L^0=L^0(\mc M,\tau) =L^0(\Om,\nu)$ is the algebra of (equivalence classes of) almost everywhere finite measurable complex-valued functions on $(\Om,\nu)$ with $\nu\{|f|>\lb\}<\ii$ for some $\lambda >0$. 
The non-increasing rearrangement of a function $f \in L^0(\Om,\nu)$ is the function $f^*$ on $(0,\ii)$ defined by (\ref{re}).

Let $(E(0,\ii),\|\cdot\|_E)$ be a symmetric function space on $((0,\ii),\mu)$. Let
$$
E(\Om, \nu)=E(\mc M, \tau)=\left\{f\in L^0(\Om, \nu):\ f^*(t)\in E(0,\ii)\right\}
$$
and
$$
\| f\|_{E(\Om)}=\| f\|_{E(\mc M)}=\| f^*(t)\|_E,  \ f\in E(\Omega,\nu).
$$
Then $(E(\Om,\nu),\|\cdot\|_{E(\Om)})$  is a symmetric function space on $(\Om,\nu)$.
Recall that this space is a fully symmetric function space if conditions
$$
f\in E(\Om,\nu), \ g\in L^0(\Om,\nu), \ \int_0^s g^*(t)dt\leq\int_0^s f^*(t)dt \ \ \forall \ s>0
$$
imply that $g\in  E(\Om,\nu)$ and $\| g\|_{E(\Om)}\leq \| f\|_{E(\Om)}$. Note that if $\nu(\Om)<\ii$, then
$E(\Om,\nu)\su L^1(\Om,\nu)$.

A net $\{ f_\al\}\su L^0(\Om, \nu)$ converges almost uniformly (a.u.) to $f\in L^0(\Om,\nu)$ if for any $\ep>0$ there
is a set $A\in\mc A$ such that $\tau(\Om\sm A)\leq \ep$ and $\| (f-f_\al)\,\chi_A\|_\ii\to 0$, where $\chi_A$ is the characteristic function of $A$.

Now we can state the following corollary of Theorems \ref{t42} and \ref{t43}.
\begin{teo}\label{t61}
Let $(\Om,\nu)$ be a Maharam measure space. Let $E=E(\Om,\nu)$ be a fully symmetric function space on $(\Om,\nu)$ such that $\mb 1\notin E$ if $\nu(\Om)=\ii$. If $\{T_t\}_{t\ge 0}\su DS^+$ is a strongly continuous semigroup in $L^1(\Om, \nu)$ and $\bt(t)$ a bounded  Besicovitch (zero-Besicovitch) function, then, given $f\in E$, the averages (\ref{e22}) converge a.u. to some $\wh f\in E$ as $t\to\ii$ (respectively, to
$\al(f)\,f$ as $t\to 0$ for some $\al(f)\in\mbb C$).
\end{teo}

\end{document}